\documentclass{amsart}
\usepackage{latexsym,amssymb,amsmath,amsthm,amscd,graphicx,ascmac,}
\usepackage{longtable}
\usepackage{enumerate}

\numberwithin{equation}{section}

\newtheorem{thm}{Theorem}[section]

\newtheorem{lem}[thm]{Lemma}
\newtheorem{prp}[thm]{Proposition}

\theoremstyle{definition}

\theoremstyle{remark}
\newtheorem{rem}[thm]{Remark}

\newcommand{\cM}{{\mathcal M}}

\newcommand{\cQ}{{\mathcal Q}}

\newcommand{\N}{{\mathbb N}}

\newcommand{\R}{{\mathbb R}}

\newcommand{\wL}{\mathrm{w}\hskip-0.6pt{L}}

\def\al{\alpha}
\def\bt{\beta}

\def\dl{\delta}

\def\Om{\Omega}

\def\0{\emptyset}
\def\6{\partial}
\def\8{\infty}

\def\lt{\left}
\def\rt{\right}

\def\ds{\displaystyle}

\newcommand{\iii}[1]{{\left\vert\kern-0.25ex\left\vert\kern-0.25ex\left\vert #1 
    \right\vert\kern-0.25ex\right\vert\kern-0.25ex\right\vert}}

 \def\XXint#1#2#3{{\setbox0=\hbox{$#1{#2#3}{\int}$}
 \vcenter{\hbox{$#2#3$}}\kern-.5\wd0}}

\begin{document}

\title{Choquet integrals, Hausdorff content and fractional operators}

\author[N.~Hatano]{Naoya Hatano}
\address{
Department of Mathematics, Chuo University, 1-13-27, Kasuga, Bunkyo-ku, Tokyo 112-8551, Japan.
}
\email{n.hatano.chuo@gmail.com}

\author[R.~Kawasumi]{Ryota Kawasumi}
\address{
Minohara 1-6-3 (B-2), Misawa, Aomori 033-0033, Japan
}
\email{rykawasumi@gmail.com}

\author[H.~Saito]{Hiroki Saito}
\address{
College of Science and Technology, Nihon University,
Narashinodai 7-24-1, Funabashi City, Chiba, 274-8501, Japan
}
\email{saitou.hiroki@nihon-u.ac.jp}

\author[H.~Tanaka]{Hitoshi~Tanaka}
\address{
Research and Support Center on Higher Education for the hearing and Visually Impaired, 
National University Corporation Tsukuba University of Technology,
Kasuga 4-12-7, Tsukuba City, Ibaraki, 305-8521 Japan
}
\email{htanaka@k.tsukuba-tech.ac.jp}

\thanks{
The first named author is financially supported by a Foundation of Research Fellows, The Mathematical Society of Japan.
The third named author is supported by 
Grant-in-Aid for Young Scientists (19K14577), 
the Japan Society for the Promotion of Science. 
The forth named author is supported by 
Grant-in-Aid for Scientific Research (C) (15K04918 and 19K03510), 
the Japan Society for the Promotion of Science.
}

\subjclass[2010]{Primary 42B25; Secondary 42B35.}

\keywords{
Choquet-Morrey spaces;
fractional integral operator;
fractional maximal operator;
Hausdorff content;
weak Choquet spaces.
}

\date{}

\begin{abstract}
It is shown that the fractional integral operator $I_{\al}$, $0<\al<n$, and the fractional maximal operator $M_{\al}$, $0\le\al<n$, are bounded on weak Choquet spaces with respect to Hausdorff content. 
We also investigate these operators on Choquet-Morrey spaces.
These results are extensions of the previous works due to Adams, Orobitg and Verdera, and Tang.
The results for the fractional integral operator $I_{\al}$ are essentially new.
\end{abstract}

\maketitle

\section{Introduction}\label{sec1}
The purpose of this paper is to study the boundedness properties of the fractional integral operator $I_{\al}$, $0<\al<n$, and the fractional maximal operator $M_{\al}$, $0\le\al<n$, on the framework of Choquet integrals with respect to Hausdorff content.

Let $n\in\N$ and $0<d\le n$.
The symbol $\cQ(\R^n)$ denotes the family of all cubes with parallel 
to coordinate axis in $\R^n$. 
The $d$-dimensional Hausdorff content of $E\subset\R^n$ 
is defined by
\[
H^d(E)
=
\inf\lt\{
\sum_{j=1}^{\8}\ell(Q_j)^d:\,
E\subset\bigcup_{j=1}^{\8}Q_j,
Q_j\in\cQ(\R^n)
\rt\},
\]
where, the infimum is taken over all coverings of the set $E$ by countable families of cubes $Q_j$, and $\ell(Q)$ stands for the side length of the cubes $Q$.
It is easily seen that
$H^n(E)$ is just the Lebesgue measure of $E$, 
which we will denote by $|E|$.
For any cube $Q$,
one has $H^d(Q)=\ell(Q)^d$.

For a non-negative function $f$,
the integral of $f$ with the respect to $H^d$ is taken in the Choquet sense: 
\[
\int_{\R^n}f\,{\rm d}H^d
=
\int_0^{\8}
H^d\lt(\lt\{
x\in \R^n:\,f(x)>t
\rt\}\rt)\,{\rm d}t.
\]
For $0<p<\8$,
the Choquet spaces $L^p(H^d)$ and 
the weak Choquet spaces $\wL^p(H^d)$ 
are to consist of all functions with 
the following properties,
\begin{align*}
\|f\|_{L^p(H^d)}
&:=
\lt(
\int_{\R^n}|f|^p\,{\rm d}H^d
\rt)^{\frac1p}<\8
\end{align*}
and
\begin{align*}
\|f\|_{\wL^p(H^d)}
&:=
\sup_{t>0}
t\,H^d\lt(\lt\{
x\in\R^n:\,|f(x)|>t
\rt\}\rt)^{\frac1p}<\8,
\end{align*}
respectively.
For $0<q\le p<\8$,
the Choquet-Morrey spaces $\cM^p_q(H^d)$
is to consist of all functions with 
the following property,
\[
\|f\|_{\cM^p_q(H^d)}
:=
\sup_{Q\in\cQ}
\ell(Q)^{\frac dp-\frac dq}
\lt(
\int_{Q}|f|^q\,{\rm d}H^d
\rt)^{\frac1q}<\8.
\]

The fractional maximal operator of order 
$\al$, $0\le\al<n$, is defined by
\[
M_{\al}f(x)
=
\sup_{Q\in\cQ(\R^n)}
\chi_{Q}(x)
\ell(Q)^{\al-n}\int_{Q}|f(y)|\,{\rm d}y,
\quad x\in\R^n,
\]
where, $\chi_{E}$ is a characteristic function of a set $E$.
If $\al=0$, then the operator $M_0$ is the usual Hardy-Littlewood maximal operator which will denote simply by $M$.
The fractional integral operator of order 
$\al$, $0<\al<n$, is defined by
\[
I_{\al}f(x)
=
\int_{\R^n}
\frac{f(y)}{|x-y|^{n-\al}}
\,{\rm d}y,\quad x\in\R^n.
\]

In the classical paper 
\cite{Orobitg-Verdera(1998)},
Orobitg and Verdera proved the following.
All the tools and ideas
that we will use
are essentially contained in this nice paper.

\begin{quotation}
For $0<d<n$, we have 
\begin{align*}
\|Mf\|_{\wL^{d/n}(H^d)}
&\lesssim
\|f\|_{L^{d/n}(H^d)},
\\
\|Mf\|_{L^p(H^d)}
&\lesssim
\|f\|_{L^p(H^d)},
\quad p>d/n.
\end{align*}
\end{quotation}

A disadvantage of the Hausdorff content $H^d$ is the following.

\begin{quotation}
It is not true that 
there exists a constant $C>0$ such that 
if $Q_1,\ldots,Q_m$ are non-overlapping dyadic cubes and $f\ge 0$, then
\[
\sum_{j=1}^m\int_{Q_j}f\,{\rm d}H^d
\le C
\int_{\bigcup_jQ_j}f\,{\rm d}H^d.
\]
This can be shown by subdividing 
the interval $[0, 1]$ in $2^m$ 
($m$ large enough) 
equal intervals, and taking $f\equiv 1$.
\end{quotation}

An advantage of the Hausdorff content $H^d$ is that the spaces $L^p(H^d)$, $0<p\le 1$, have the block decomposition and, as a corollary, the following assertion holds 
(see Theorem 2.3 in \cite{Saito-Tanaka-Watanabe2020}).

\begin{quotation}
Suppose that $T$ is a subadditive operator.
Let $0<d,\dl<n$, $0<p\le 1$ and $q\ge p$. Then 
the following statements are equivalent\text{:}

\begin{itemize}
\item[{\rm(a)}]
The inequality 
$\ds\|Tf\|_{L^q(H^{\dl})}
\le C_1
\|f\|_{L^p(H^d)}$ 
holds;
\item[{\rm(b)}]
The testing inequality
$\ds\|T\chi_{Q}\|_{L^q(H^{\dl})}
\le C_2\ell(Q)^{\frac1p}$
holds for any $Q\in\cQ(\R^n)$.
\end{itemize}

\noindent
Moreover, 
the least possible constants $C_1$ and $C_2$ are equivalent.
\end{quotation}

Under this advantage, by using an easy testing inequality,
the following result had verified 
(See \cite{Adams1998}).

\begin{quotation}
For $d/n<p<d/\al$, we have
\begin{equation}\label{1.1}
\|M_{\al}f\|_{L^p(H^{d-\al p})}
\lesssim
\|f\|_{L^p(H^d)}.
\end{equation}
For $0<d\le n$, we have
\begin{equation}\label{1.2}
\|M_{\al}f\|_{\wL^{d/n}(H^{d-\al d/n})}
\lesssim
\|f\|_{L^{d/n}(H^d)}.
\end{equation}
\end{quotation}

The simple trick
why the advantage has influence on 
the cases $p>1$ is that, 
for the fractional maximal operator $M_{\al}$,
one has the pointwise estimate
\[
M_{\al}f(x)^p\le M_{\al p}[f^p](x),
\quad x\in\R^n,
\]
and is reduced to the case $p=1$.
This trick no longer works for the fractional integral operator $I_{\al}$.
In this paper, we establish the following, 
the difficulty overcomes by clever use of 
Hedberg's trick (see Lemma \ref{lem4.1}) 
which is due to the first author N.~Hatano.

\begin{thm}\label{thm1.1}
\begin{itemize}
\item[{\rm(i)}]
Let $0<d\le n$ and $0\le \al<n$.
Suppose that $d/n\le r<p<d/\al$ and
\[
\frac{d-\al r}{q}
=
\frac{d-\al p}{p}.
\]
Then
\[
\|M_{\al}f\|_{\wL^q(H^{d-\al r})}
\lesssim
\|f\|_{\wL^p(H^d)};
\]
\item[{\rm(ii)}]
Let $0<d<n$ and $0<\al<n$.
Suppose that $d/n<r<p<d/\al$ and
\[
\frac{d-\al r}{q}
=
\frac{d-\al p}{p}.
\]
Then
\[
\|I_{\al}f\|_{\wL^q(H^{d-\al r})}
\lesssim
\|f\|_{\wL^p(H^d)};
\]
\item[{\rm(iii)}]
Let $0<d<n$ and $0<\al<n$.
Suppose that $d/n<r<p<d/\al$ and
\[
\frac{d-\al r}{q}
=
\frac{d-\al p}{p}.
\]
Then 
\[
\|I_{\al}f\|_{L^q(H^{d-\al r})}
\lesssim
\|f\|_{L^p(H^d)}.
\]
\end{itemize}
\end{thm}

\begin{rem}
Taking into account \eqref{1.1}, 
one may expect that 
Theorem \ref{thm1.1} (i) holds for 
the case $q=r=p$.
This would be impossible because then,
as a special case,
we must have the false inequality:
\[
\|M_{\al}f\|_{\wL^1(H^{n-\al})}
\lesssim
\|f\|_{\wL^1(H^n)}.
\]
\end{rem}

The key ingredient in the proof of 
Theorem \ref{thm1.1} (i) 
is the following Kolmogorov-type inequality.

\begin{quotation}
For any measurable set $E\subset\R^n$,
\[
H^d(E)^{\frac1p-\frac1q}
\|f\chi_{E}\|_{L^q(H^d)}
\le
\lt(\frac{q}{p-q}\rt)^{1/p}
\|f\chi_{E}\|_{\wL^p(H^d)},
\quad
0<q<p<\8.
\]
\end{quotation}

This means that
the weak $L^p$ integrability
implies the local $L^q$ integrability
provided that $q<p$.
In Kolmogorov-type inequality, 
taking the supremum over all measurable set,
we can obtain the reverse inequality 
and, as a consequence, 
the weak Choquet norm
can be written as the following formula 
(see Proposition \ref{prp2.4}).

\begin{quotation}
For $0<q<p<\8$, we have
\[
\|f\|_{\wL^p(H^d)}
\le
\sup_{0<H^d(E)<\8}
H^d(E)^{\frac1p-\frac1q}
\|f\chi_{E}\|_{L^q(H^d)}
\le
\lt(\frac{q}{p-q}\rt)^{\frac1p}
\|f\|_{\wL^p(H^d)}.
\]
\end{quotation}

We notice that 
the parameter $q$ does not affect 
the set $\wL^p(H^d)$.
On the other hand, 
in the above supremum, 
if one restricts to 
the cube $Q\in\cQ(\R^n)$ 
instead of the general set $E\subset\R^n$, 
then one gets Morrey spaces and 
one no longer ignores the influence of the second parameter $q$.
In this paper, we establish the following

\begin{thm}\label{thm1.2}
\begin{itemize}
\item[{\rm(i)}]
({\cite[Theorem 2]{Tang2011}})
Let $0<d\le n$ and $0\le \al<n$.
Suppose that 
$d/n<r\le p<d/\al$ and
\[
\frac{d-\al r}{q}
=
\frac{d-\al p}{p}.
\]
Then 
\[
\|M_{\al}f\|_{\cM^q_r(H^{d-\al r})}
\lesssim
\|f\|_{\cM^p_r(H^d)};
\]
\item[{\rm(ii)}]
Let $0<d<n$ and $0<\al<n$.
Suppose that $d/n<r<p<d/\al$ and
\[
\frac{d-\al r}{q}
=
\frac{d-\al p}{p}.
\]
Then 
\[
\|I_{\al}f\|_{\cM^q_r(H^{d-\al r})}
\lesssim
\|f\|_{\cM^p_r(H^d)}.
\]
\end{itemize}
\end{thm}

This paper is organized as follows.
In Section 2,
We give a proof of Kolmogorov-type inequality
and summarize some elementary properties.
In Section 3,
we give the proof of theorems.
In section 4, as an appendix, 
we gather some results.

Throughout this paper,
we use the following notation.
Let $X$ and $Y$ be the normed spaces with 
$\|\cdot\|_{X}
\lesssim
\|\cdot\|_{Y}$, 
then we write 
$X\hookleftarrow Y$ or 
$Y\hookrightarrow X$ 
(sometimes called the embedding). 
If $X\hookleftarrow Y$ and 
$X\hookrightarrow Y$, then 
we write $X=Y$.

For the quantities $A$ and $B$,
if $A\le C B$, then 
we write $A\lesssim B$ or 
$B\gtrsim A$ and 
if $A\lesssim B$ and $A\gtrsim B$, then 
we write $A\sim B$.

\section{Preliminaries}\label{sec2}
We use the following fact of the Hausdorff content, 
due to Orobitg and Verdera in 
{\cite[Lemma~3]{Orobitg-Verdera(1998)}}.

\begin{lem}\label{lem2.1}
Let $0<d\le n$ and 
$0<p<\8$. Then, for 
$1\le\theta\le n/d$,
\begin{equation}\label{2.1}
\|f\|_{L^{\theta p}(H^{\theta d})}
\le
\theta^{\frac1{\theta p}}
\|f\|_{L^p(H^d)}.
\end{equation}
\end{lem}

\begin{proof}
It follows that
\begin{align*}
\|f\|_{L^{\theta p}(H^{\theta d})}^{\theta p}
&=
\int_0^{\8}
H^{\theta d}(\{|f|^{\theta p}>t\})
\,{\rm d}t\qquad(t=s^{\theta})
\\ &=
\theta
\int_0^{\8}
H^{\theta d}(\{|f|^p>s\})
s^{\theta-1}
\,{\rm d}s.
\end{align*}
Since,
\begin{equation}\label{2.2}
\lt(\sum_j a_j^q\rt)^{\frac1q}
\le
\lt(\sum_j a_j^p\rt)^{\frac1p},
\quad a_j\ge 0,\quad 0<p<q<\8,
\end{equation}
we obtain
\[
H^{\theta d}(\{|f|^p>s\})
\le
\lt(H^d(\{|f|^p>s\})\rt)^{\theta}.
\]
Thus,
\begin{align*}
&\le
\theta
\int_0^{\8}
\lt(H^d(\{|f|^p>s\})\rt)^{\theta}
s^{\theta-1}
\,{\rm d}s
\\ &=
\theta
\int_0^{\8}
\lt(H^d(\{|f|^p>s\})s\rt)^{\theta-1}
H^d(\{|f|^p>s\})
\,{\rm d}s.
\end{align*}
We have that
\[
H^d(\{|f|^p>s\})s
=
\int_{\{|f|^p>s\}}s\,{\rm d}H^d
\le
\int_{\{|f|^p>s\}}|f|^p\,{\rm d}H^d
\le
\int_{\R^n}|f|^p\,{\rm d}H^d.
\]
Consequently,
\[
\le
\theta
\|f\|_{L^p(H^d)}^{\theta p}.
\]
This completes the proof.
\end{proof}

We summarize elementary properties of 
the weak Choquet spaces.

\begin{prp}\label{prp2.2}
\begin{enumerate}[{\rm (W1)}]
\item (Chebyshev's inequality) 
For $p>0$, then 
$L^p(H^d)\hookrightarrow \wL^p(H^d)$ 
holds;
\item
If $0<d\le\theta d\le n$ and 
$0<p<\8$, then
\[
\wL^p(H^d)
\hookrightarrow
\wL^{\theta p}(H^{\theta d});
\]
\item
If $d=n$, then 
the weak Choquet spacse $\wL^p(H^d)$ 
are the usual weak $L^p$ spaces $\wL^p(\R^n)$.
\end{enumerate}
\end{prp}

\begin{proof}
The assertions (W1) and (W3) 
can be shown immediately from definition.
To prove (W2),
we observe that
\begin{align*}
\|f\|_{\wL^{\theta p}(H^{\theta d})}
&=
\sup_{t>0}
t\,H^{\theta d}(\{
x\in\R^n:\,|f(x)|>t
\})^{1/(\theta p)}
\\ &\le
\sup_{t>0}
t\,H^d(\{
x\in\R^n:\,|f(x)|>t
\})^{1/p}
=
\|f\|_{\wL^p(H^d)},
\end{align*}
where, we have used \eqref{2.2}.
\end{proof}

We prove the Kolmogorov-type inequality.

\begin{quotation}
For $0<p,r\le\8$, 
we define the Lorentz quasinorm 
\[
\|f\|_{L^{p,r}(H^d)}
=
\lt(
\int_0^{\8}
(t^pH^d(\{|f|>t\}))^{\frac rp}
\frac{{\rm d}t}{t}
\rt)^{\frac1r}
\]
and we denote the Lorentz space 
$L^{p,r}(H^d)$ by 
the set of all functions for which the above quasinorms are finite.
\end{quotation}

\begin{prp}\label{prp2.3}
Let $0<d\le n$, 
$0<q<p<\8$ and 
$E$ be any measurable set. 

\begin{itemize}
\item[{\rm(a)}]
if $q\le r\le p$, then
\[
H^d(E)^{\frac1p-\frac1q}
\lt(\int_{E}|f|^q\,{\rm d}H^d\rt)^{\frac1q}
\lesssim
q^{\frac1r}
\|f\chi_{E}\|_{L^{p,q}(H^d)};
\]
\item[{\rm(b)}]
If $p<r$, then
\begin{equation}\label{2.3}
H^d(E)^{\frac1p-\frac1q}
\lt(\int_{E}|f|^q\,{\rm d}H^d\rt)^{\frac1q}
\lesssim
\lt(q\lt(\frac{r-p}{r(p-q)}\rt)^{\frac{r-p}{r}}\rt)^{\frac1p}
\|f\chi_{E}\|_{L^{p,r}(H^d)}.
\end{equation}
\end{itemize}
\end{prp}

\begin{proof}
\noindent{\bf(a)}\quad
Let $E$ be an arbitrary measurable set 
satisfying $0<H^d(E)<\8$.
For any $A>0$,
we observe that
\begin{align*}
&\int_{E}|f|^q\,{\rm d}H^d
=
\int_0^{\8}
H^d(E\cap\{|f|^q>t\})\,{\rm d}t
\\ &= 
q\int_0^{\8}
t^{q-1}
H^d(E\cap\{|f|>t\})\,{\rm d}t
\\ &\le
q\int_0^{A}
t^{q-1}H^d(E)\,{\rm d}t
+
q\int_{A}^{\8}
t^{q-r}
\cdot
t^rH^d(E\cap\{|f|>t\})^{\frac rp}
\cdot
H^d(E\cap\{|f|>t\})^{1-\frac rp}
\,\frac{{\rm d}t}{t}
\\ &\le
H^d(E)^{1-\frac rp}
\lt(
A^qH^d(E)^{\frac rp}
+
qA^{q-r}
\|f\chi_{E}\|_{L^{p,r}(H^d)}^r
\rt).
\end{align*}
If we take 
$A^r
=
q\|f\chi_{E}\|_{L^{p,r}(H^d)}^r
H^d(E)^{-r/p}$, then 
the two terms on the right-hand side of the above inequality balance. Thus,
\begin{align*}
\int_{E}|f|^q\,{\rm d}H^d
&\le 2
H^d(E)^{1-\frac rp}
\lt(q\|f\chi_{E}\|_{L^{p,r}(H^d)}^r\rt)^{\frac qr}
H^d(E)^{\frac rp-\frac qp}
\\ &\le 2q^{\frac qr}
H^d(E)^{1-\frac qp}
\|f\chi_{E}\|_{L^{p,r}(H^d)}^q.
\end{align*}
and this implies
\[
H^d(E)^{\frac1p-\frac1q}
\lt(\int_{E}|f|^q\,{\rm d}H^d\rt)^{\frac1q}
\lesssim
q^{\frac1r}
\|f\chi_{E}\|_{L^{p,q}(H^d)}.
\]

\noindent{\bf(b)}\quad
It follows that
\begin{align*}
&\int_{A}^{\8}
t^{q-1}
H^d(E\cap\{|f|>t\})\,{\rm d}t
\\ &\le
\int_{A}^{\8}
t^qH^d(\{|f|>t\})
\,\frac{{\rm d}t}{t}
\\ &=
\int_{A}^{\8}
t^{q-p}
\cdot
t^pH^d(\{|f|>t\})
\,\frac{{\rm d}t}{t}
\\ &\le
\lt(
\int_{A}^{\8}
t^{(q-p)\frac{r}{r-p}}
\,\frac{{\rm d}t}{t}
\rt)^{\frac{r-p}{r}}
\lt(
\int_{A}^{\8}
(t^pH^d(\{|f|>t\}))^{\frac rp}
\,\frac{{\rm d}t}{t}
\rt)^{\frac pr}
\\ &\le
\lt(\frac{r-p}{r(p-q)}\rt)^{\frac{r-p}{r}}
A^{q-p}
\|f\chi_{E}\|_{L^{p,r}(H^d)}^p.
\end{align*}
This implies
\[
\int_{E}|f|^q\,{\rm d}H^d
\lesssim
A^qH^d(E)
+
q\lt(\frac{r-p}{r(p-q)}\rt)^{\frac{r-p}{r}}
A^{q-p}
\|f\chi_{E}\|_{L^{p,r}(H^d)}^p.
\]
The same as the above,
\[
\int_{E}|f|^q\,{\rm d}H^d
\lesssim
\lt(q\lt(\frac{r-p}{r(p-q)}\rt)^{\frac{r-p}{r}}\rt)^{\frac qp}
H^d(E)^{1-\frac qp}
\|f\chi_{E}\|_{L^{p,r}(H^d)}^q,
\]
which yields \eqref{2.3}.
\end{proof}

\begin{rem}
Taking $E=Q\in\cQ(\R^n)$ and 
$r=\8$ in \eqref{2.3}, we have 
\[
\ell(Q)^{\frac dp-\frac dq}
\|f\chi_{Q}\|_{L^q(H^d)}
\lesssim
\|f\chi_{Q}\|_{L^{p,\8}(H^d)}.
\]
For an arbitrary $r$ with 
$p\le r<\8$, multiplying 
$\ds\ell(Q)^{\frac dr-\frac dp}$
to the both sides of the above inequality,
we have that
\[
\ell(Q)^{\frac dr-\frac dq}
\|f\chi_{Q}\|_{L^q(H^d)}
\lesssim
\ell(Q)^{\frac dr-\frac dp}
\|f\chi_{Q}\|_{L^{p,\8}(H^d)}.
\]
This means that
\begin{equation}\label{2.4}
\|f\|_{\cM^r_q(H^d)}
\lesssim
\|f\|_{\cM^r_{p,\8}(H^d)}
\end{equation}
for $0<q<p\le r<\8$. Here, 
\[
\|f\|_{\cM^r_{p,\8}(H^d)}
:=
\sup_{Q\in\cQ(\R^n)}
\ell(Q)^{\frac dr-\frac dp}
\|f\chi_{Q}\|_{L^{p,\8}(H^d)}.
\]
That \eqref{2.4} is interesting because
one always has 
\[
\|f\|_{\cM^r_q(H^d)}
\le
\|f\|_{\cM^r_p(H^d)},
\]
and
\[
\|f\|_{\cM^r_{p,\8}(H^d)}
\le
\|f\|_{\cM^r_p(H^d)}.
\]
That \eqref{2.4} was found in~\cite{Gunawan-Hakim-Limanta-Masta(2017)} and  the book \cite{Sa}.
\end{rem}

In the case $r=\8$ in 
Proposition \ref{prp2.3},
we can show the reverse inequality.

\begin{prp}\label{prp2.4}
Let $0<d\le n$ and 
$0<q<p<\8$. Then 
\[
\|f\|_{\wL^p(H^d)}
\le
\sup_{0<H^d(E)<\8}
H^d(E)^{\frac1p-\frac1q}
\lt(
\int_{E}|f|^q\,{\rm d}H^d
\rt)^{\frac 1q}
\le
\lt(\frac{q}{p-q}\rt)^{1/p}
\|f\|_{\wL^p(H^d)}.
\]
\end{prp}

\begin{proof}
The second inequality holds by 
letting $r=\8$ in \eqref{2.3}.
We prove the reverse inequality.
For all $t>0$,
\begin{align*}
&
t\,H^d(\{
x\in\R^n:\,|f(x)|>t
\})^{\frac1p}
\\ &=
H^d(\{
x\in\R^n:\,|f(x)|>t
\})^{\frac1p-\frac1q}
\cdot
t\,H^d(\{
x\in\R^n:\,|f(x)|>t
\})^{\frac1q}
\\ &\le
H^d(\{
x\in\R^n:\,|f(x)|>t
\})^{\frac1p-\frac1q}
\lt(
\int_{\{|f|>t\}}|f|^q\,{\rm d}H^d
\rt)^{\frac1q}.
\end{align*}
Hence,
\[
t\,H^d(\{
x\in\R^n:\,\,|f(x)|>t
\})^{\frac1p}
\le
\sup_{0<H^d(E)<\8}
H^d(E)^{\frac1p-\frac1q}
\lt(
\int_{E}|f|^q\,{\rm d}H^d
\rt)^{\frac1q},
\]
as desired.
\end{proof}

We summarize elementary properties of the Choquet-Morrey spaces.

\begin{prp}\label{prp2.5}
\begin{enumerate}[{\rm (M1)}]
\item
If $0<p=q<\8$, then 
$\cM_p^p(H^d)=L^p(H^d)$ holds;
\item
If $0<q<p<\8$, then 
$\wL^p(H^d)
\hookrightarrow
\cM_q^p(H^d)$ holds;
\item
If $d=n$, then 
the Choquet-Morrey  spaces
$\cM^p_q(H^d)$ are the usual Morrey spaces $\cM^p_q(\R^n)$;
\item
If $0<q\le r\le p<\8$, then 
$\cM^p_r(H^d)
\hookrightarrow 
\cM^p_q(H^d)$ holds;
\item
If $0<d\le\theta d\le n$ and 
$0<q\le p<\8$, then
\[
\cM^p_q(H^d)
\hookrightarrow
\cM^{\theta p}_{\theta q}(H^{\theta d}).
\]
\end{enumerate}
\end{prp}

\begin{proof}
The assertions (M1)--(M4) are obvious.
To show (M5) there holds, 
by using \eqref{2.1} and 
setting
$\tilde{d}=\theta d$, 
$\tilde{p}=\theta p$ and
$\tilde{q}=\theta q$, 
\[
\ell(Q)^{\frac{\tilde{d}}{\tilde{p}}-\frac{\tilde{d}}{\tilde{q}}}
\|f\chi_{Q}\|_{L^{\tilde{q}}(H^{\tilde{d}})}
\le\theta
\ell(Q)^{\frac dp-\frac dq}
\|f\chi_{Q}\|_{L^q(H^d)}.
\]
\end{proof}

It is well known that, 
due to Hedberg \cite{Hedberg72}, 
for $1\le p<\8$, the pointwise estimate
\[
|I_{\al}f(x)|
\lesssim
\|f\|_{L^p(\R^n)}^{\frac{\al p}{n}}
Mf(x)^{1-\frac{\al p}{n}}.
\]
To prove the boundedness of the fractional integral operators 
on weak Choquet and Choquet-Morrey spaces 
with Hausdorff content, 
we extend Hedberg's inequality to the following.

\begin{lem}\label{lem4.1}
Let $0\le\bt<\al<n$ and 
$d/n\le p<q<\8$ with 
\[
\frac{d-\bt p}{q}
=
\frac{d-\al p}{p}.
\]
Then 
\[
|I_{\al}f(x)|
\lesssim
\left(\frac dp-\al\right)^{\frac pq-1}
(\al-\bt)^{-\frac pq}
\|f\|_{\cM^p_{d/n}(H^d)}^{1-\frac pq}
M_{\bt}f(x)^{\frac pq},
\quad x\in\R^n.
\]
\end{lem}

\begin{proof}
Fix $r>0$. We decompose
\begin{align*}
|I_{\al}f(x)|
&\le
\sum_{j=-\8}^{\8}
\int_{2^jr \le |x-y| < 2^{j+1}r} 
\frac{|f(y)|}{|x-y|^{n-\al}}
\,{\rm d}y
\\ &\le
\sum_{j=-\8}^{\8}
\frac{1}{(2^jr)^{n-\al}}
\int_{|x-y|<2^{j+1}r}
|f(y)|\,{\rm d}y
\\ &=
\sum_{j=-\8}^{-1}
\frac{1}{(2^jr)^{n-\al}}
\int_{|x-y|<2^{j+1}r}
|f(y)|\,{\rm d}y
+
\sum_{j=0}^{\8}
\frac{1}{(2^jr)^{n-\al}}
\int_{ |x-y|<2^{j+1}r}
|f(y)|\,{\rm d}y
\\ &=:
J_1+J_2.
\end{align*}
We estimate
\begin{align*}
J_1
\lesssim
\sum_{j=-\8}^{-1}
(2^jr)^{\al-\bt}
M_{\bt}f(x)
\lesssim
\frac{r^{\al-\bt}}{\al-\bt}
M_{\bt}f(x),
\end{align*}
where, we have used 
$2^{\al-\bt}-1 \gtrsim \al-\bt$.
By \eqref{2.1}, we estimate
\begin{align*}
J_2
\lesssim
\sum_{j=1}^{\8}
(2^jr)^{\al-\frac dp}
\|f\|_{\cM^p_{d/n}(H^d)}
\lesssim
\frac{r^{\al-\frac dp}}{\frac dp-\al}
\|f\|_{\cM^p_{d/n}(H^d)},
\end{align*}
where, we have used
$2^{\frac dp-\al}-1 \gtrsim \frac dp-\al$.
Then, 
taking the optimal quantity $r>0$ 
in the previous inequalities and 
noticing that
\[
\frac pq
=
\frac{d-\al p}{d-\bt p},
\]
we obtain
\[
|I_{\al}f(x)|
\lesssim
\lt(\frac dp-\al\rt)^{\frac pq-1}
(\al-\bt)^{-\frac pq}
\|f\|_{\cM^p_{d/n}(H^d)}^{1-\frac pq}
M_{\bt}f(x)^{\frac pq},
\]
as desired.
\end{proof}

\section{Proof of the theorems}\label{sec3}
In what follows we prove the theorems.

\subsection*{
Proof of Theorem \ref{thm1.1} (i)
}
Set $\dl=d-\al r$ and, for $t>0$, 
\[
\Om(M_{\al}f;t)
:=
\{x\in\R^n:\,M_{\al}f(x)>t\}.
\]
Then we see that
\[
\Om(M_{\al}f;t)
=
\{x\in\R^n:\,
M_{\al}[f\chi_{\Om(M_{\al}f;t)}](x)>t\}.
\]
Thus, there holds 
\begin{align*}
t\,H^{\dl}(\Om(M_{\al}f;t))^{\frac1q}
&=
H^{\dl}(\Om(M_{\al}f;t))^{\frac1q-\frac1r}
\cdot
t\,H^{\dl}(\Om(M_{\al}f;t))^{\frac1r}
\\ &\lesssim
H^{\dl}(\Om(M_{\al}f;t))^{\frac1q-\frac1r}
\cdot
\|f\chi_{\Om(M_{\al}f;t)}\|_{L^r(H^d)},
\end{align*}
where, we have used \eqref{1.1} or 
\eqref{1.2}, when $r=d/n$.
By the use of the Kolmogorov-type inequality \eqref{2.3},
we need only verify that, 
for a compact set $E\subset\R^n$,
\[
H^{\dl}(E)^{\frac1r-\frac1q}
\ge
H^d(E)^{\frac1r-\frac1p}.
\]
By the fact $\dl<d$ and \eqref{2.2}
we have that
\[
H^{\dl}(E)^{\frac1r-\frac1q}
\ge
H^d(E)^{\frac{\dl}{d}(\frac1r-\frac1q)},
\]
and
\[
\frac{\dl}{d}\lt(\frac1r-\frac1q\rt)
=
\frac1r-\frac\al d
-
\frac{d-\al p}{dp}
=
\frac1r-\frac1p.
\]
This completes the proof.
\qed

\subsection*{
Proof of Theorem \ref{thm1.1} (ii)
}
Thanks to the continuity,
one can choose 
$1<\theta<n/d$ and 
$\frac rp \al<\bt<\al$ 
so that
\[
\theta(d-\bt p)
=
d-\al r
=
d-\frac rp \al p.
\]
Setting $\dl=\theta d$ and 
$u=\theta p$, one has
\[
\dl-\bt u=d-\al r.
\]
This equation and 
\[
\frac{d-\al r }{q}
=
\frac{d-\al p}{p}
\]
immediately imply
\[
\frac{\dl-\bt u}{q}
=
\frac{d-\al r}{q}
=
\frac{d-\al p}{p}
=
\frac dp-\al
=
\frac\dl u-\al
=
\frac{\dl-\al u}{u}.
\]
Noticing $\bt<\al$, 
we see that $p<u<q$.
The equation
\[
\frac{\dl-\bt u}{q}
=
\frac{\dl-\al u}{u}
\]
and Lemma \ref{lem4.1} yield
\[
|I_{\al}f(x)|
\lesssim
\|f\|_{\cM^u_{\dl/n}(H^{\dl})}^{1-\frac uq}
M_{\bt}f(x)^{\frac uq},
\quad x\in\R^n.
\]
The equations
\[
\frac{d-\al r}u
=
\frac{\dl-\bt u}u
=
\frac{d-\al p}p
\]
and Theorem \ref{thm1.1} (i) yield
\[
\|(M_{\bt}f)^{u/q}\|_{\wL^q(H^{d-\al r})}
=
\|M_{\bt}f\|_{\wL^u(H^{d-\al r})}^{u/q}
\lesssim
\|f\|_{\wL^p(H^d)}^{u/q}.
\]
Since we always have 
\[
\|f\|_{\cM^u_{\dl/n}(H^{\dl})}
\lesssim
\|f\|_{\cM^p_{d/n}(H^d)}
\lesssim
\|f\|_{\wL^p(H^d)},
\]
which completes the proof.
\qed

\subsection*{
Proof of Theorem \ref{thm1.1} (iii)
}
Setting $\bt=\frac rp \al$, we have
\[
\frac{d-\bt p}{q}
=
\frac{d-\al r}{q}
=
\frac{d-\al p}{p}.
\]
It follows from Lemma \ref{lem4.1} that
\[
\|I_{\al}f\|_{L^q(H^{d-\bt p})}
\lesssim
\|f\|_{\cM^p_{d/n}(H^d)}^{1-\frac pq}
\|(M_{\bt}f)^{\frac pq}\|_{L^q(H^{d-\bt p})}.
\]
By \eqref{1.1} we have that
\[
\|M_{\bt}f\|_{L^p(H^{d-\bt p})}
\lesssim
\|f\|_{L^p(H^d)},
\]
and, by the inclusion property, that
\[
\|f\|_{\cM^p_{d/n}(H^d)}
\lesssim
\|f\|_{L^p(H^d)}.
\]
These complete the proof.
\qed

\subsection*{
Proof of Theorem \ref{thm1.2} (ii)
}
This theorem can be proved 
the same manner as that of 
Theorem \ref{1.1} (ii).
\qed

\section{Appendix}\label{sec4}
In what follows we gather the results 
which hold by our theorems and the embedding lemma (Lemma \ref{lem2.1}).

\begin{prp}
Let $0<d,\dl<n$ and $0\le \al<n$.
Suppose that $q>p$, $d/n<p<d/\al$ and 
\[
\frac{\dl}{q}=\frac{d}{p}-\al.
\]
Then 
\[
\|M_{\al}f\|_{\wL^q(H^{\dl})}
\lesssim
\|f\|_{\wL^p(H^d)}.
\]
\end{prp}

\begin{proof}
If we are given $q$, $q>p$, and 
$\dl$, $0<\dl\le n$, such that
\[
\frac{\dl}{q}
=
\frac{d-\al p}{p}
=
\frac{d}{p}-\al.
\]
Then, because $\dl>d-\al p$,
we can choose
$u$, $r$, $d/n<r<p$, and 
$\theta>1$ so that
\[
q=\theta u,
\quad
\dl=\theta(d-\al r).
\]
Since we have
\[
\frac{d-\al r}{u}
=
\frac{d-\al p}{p},
\]
Theorem \ref{thm1.1} (i) yields
\[
\|M_{\al}f\|_{\wL^u(H^{d-\al r})}
\lesssim
\|f\|_{\wL^p(H^d)},
\]
which gives, by the embedding 
(Proposition \ref{prp2.2}), that
\[
\|M_{\al}f\|_{\wL^q(H^{\dl})}
\lesssim
\|f\|_{\wL^p(H^d)}.
\]
This completes the proof.
\end{proof}

\begin{prp}\label{Spanne}
Let $0<d,\dl\le n$ and $0\le \al<n$.
Suppose that 
$d/n<r\le p<d/\al$, $q\ge p$, 
$\dl\ge d-\al r$, $s\ge r$ and
\[
\frac{\dl}{q}=\frac{d}{p}-\al,
\quad
\frac{\dl}{s}=\frac{d}{r}-\al.
\]
Then 
\[
\|M_{\al}f\|_{\cM^q_s(H^{\dl})}
\lesssim
\|f\|_{\cM^p_r(H^d)}.
\]
\end{prp}

\begin{proof}
If we are given $q$, $q\ge p$, and 
$\dl$, $n\ge\dl\ge d-\al r$, such that
\[
\frac{\dl}{q}
=
\frac{d}{p}-\al.
\]
Then, we can choose
$u$ and $\theta\ge1$ so that
\[
q=\theta u
\quad\text{and}\quad
\dl=\theta(d-\al r).
\]
Since we have
\[
\frac{d-\al r}{u}
=
\frac{d-\al p}{p},
\]
it follows from Theorem \ref{thm1.2} (i) that
\[
\|M_{\al}f\|_{\cM^u_r(H^{d-\al r})}
\lesssim
\|f\|_{\cM^p_r(H^d)}.
\]
Since we have also 
\[
\frac dr-\al
=
\frac{d-\al r}{r}
=
\frac{\dl}{\theta r},
\]
we must $s=\theta r$. 
Thus, by the inclusion property 
(Proposition \ref{prp2.5}) 
we obtain
\[
\|M_{\al}f\|_{\cM^q_s(H^{\dl})}
\lesssim
\|M_{\al}f\|_{\cM^u_r(H^{d-\al r})}
\lesssim
\|f\|_{\cM^p_r(H^d)},
\]
as desired.
\end{proof}

The case $d=\dl=n$ of 
Proposition \ref{Spanne}
was first studied in unpublished work of Spanne 
(for the fractional integral operator)
and Peetre published in \cite[Theorem 5.4]{Peetre69}.
So, 
Proposition \ref{Spanne} is 
the Spanne-type theorem.
After that, 
the Spanne-type condition 
is generalized as
\[
\frac1q=\frac1p-\frac \al n,
\quad
\frac{s}{q}=\frac{r}{p}
\]
by Adams \cite{Adams75}, 
Chiarenza and Frasca \cite{ChFr87}.

\begin{prp}
Let $0<d\le n$, $0<\al<n$, 
$d/n<r\le p<\8$ and
$0<s\le q<\8$.
Suppose that
\[
\frac1q=\frac1p-\frac\al d,
\quad
\frac sq=\frac rp.
\]
Then 
\[
\|I_{\al}f\|_{\cM^q_s(H^d)}
\lesssim
\|f\|_{\cM^p_r(H^d)}.
\]
\end{prp}

\begin{proof}
By taking $\bt=0$ in Lemma \ref{lem4.1},
we have
\[
|I_{\al}f(x)|
\lesssim
\|f\|_{\cM^p_{d/n}(H^d)}^{1-\frac pq}
Mf(x)^{\frac pq}.
\]
The boundedness of $M$ on the Morrey spaces
enable us that
\begin{align*}
\|I_{\al}f\|_{\cM^q_s(H^d)}
&\lesssim
\|f\|_{\cM^p_{d/n}(H^d)}^{1-\frac pq}
\|(Mf)^{p/q}\|_{\cM^q_s(H^d)}
\\ &=
\|f\|_{\cM^p_{d/n}(H^d)}^{1-\frac pq}
\|Mf\|_{\cM^p_r(H^d)}^{p/q}
\\ &\lesssim
\|f\|_{\cM^p_{d/n}(H^d)}^{1-\frac pq}
\|f\|_{\cM^p_r(H^d)}^{\frac pq}
\lesssim
\|f\|_{\cM^p_r(H^d)}.
\end{align*}
where, in the last inequality we have used 
Proposition \ref{prp2.5}.
\end{proof}


\begin{thebibliography}{999}

\bibitem{Adams75}
D. R. Adams,
A note on Riesz potentials,
Duke Math. J. 42 (1975), no. 4, 765--778.
\bibitem{Adams1998}
D. R. Adams, Choquet integrals in potential theory, Publ. Mat. 42 (1998) 3--66.

\bibitem{ChFr87}
F.~Chiarenza and M.~Frasca,
Morrey spaces and Hardy-Littlewood maximal function,
Rend. Mat., 7 (1987), 273--279.


\bibitem{Gunawan-Hakim-Limanta-Masta(2017)} 
H. Gunawan, D. I. Hakim, K. M. Limanta and A. A. Masta, Inclusion properties
of generalized Morrey spaces, Math. Nachr. 290 (2017), No. 2-3, 332–340.


\bibitem{Hedberg72}
L. I. Hedberg,
On certain convolution inequalities,
Proc. Amer. Math. Soc. 36 (1972), 505--510.


\bibitem{Orobitg-Verdera(1998)}
J. Orobitg and J. Verdera, Choquet integrals, Hausdorff
content and the Hardy-Littlewood maximal operator, Bull.
London Math. Soc., 30 (1998), no. 2, 145--150.

\bibitem{Peetre69}
J. Peetre,
On the theory of ${\mathcal L}_{p,\lambda}$ spaces,
Jour. Funct. Anal. {\bf 4} (1969), 71--87.

\bibitem{Saito-Tanaka-Watanabe2020} 
H. Saito, H. Tanaka, T. Watanabe, 
Block decomposition and weighted Hausdorff content,
Canad. Math. Bull. 63 (2020), no. 1, 141--156.

\bibitem{Sa} 
Y.~Sawano,
G.~Di~Fazio and D.~I.~Hakim,
{\it
Morrey Spaces: Introduction and 
Applications to Integral Operators and 
PDE's}, 
Volume I 
Chapman \& Hall/CRC Monographs and Research Notes in Mathematics), 2020/9/17.

\bibitem{Tang2011} 
L. Tang,
Choquet integrals, weighted Hausdorff content and maximal operators,
Georgian Math. J. vol. 18 (2011), no. 3, pp.~587--596.

\end{thebibliography}
\end{document}